\documentclass[12pt,psamsfonts]{amsart}
\usepackage[utf8]{inputenc}

\usepackage{amsmath}
\usepackage{amsthm}
\usepackage{amssymb}
\usepackage{amscd}
\usepackage{amsfonts}
\usepackage{amsbsy}
\usepackage{graphicx}
\usepackage[dvips]{psfrag}
\usepackage{array}
\usepackage{color}
\usepackage{epsfig}
\usepackage{url}

\usepackage{overpic}
\usepackage{epstopdf}

\newcommand{\R}{\ensuremath{\mathbb{R}}}
\newcommand{\N}{\ensuremath{\mathbb{N}}}

\newcommand{\CC}{\mathcal{C}}

\newcommand{\CU}{\ensuremath{\mathcal{U}}}

\newcommand{\ov}{\overline}
\newcommand{\la}{\lambda}
\newcommand{\g}{\gamma}
\newcommand{\G}{\Gamma}
\newcommand{\T}{\theta}
\newcommand{\f}{\varphi}
\newcommand{\al}{\alpha}
\newcommand{\be}{\beta}

\def\p{\partial}

\newtheorem {theorem} {Theorem}%[section]
\newtheorem {definition} {Definition}
\newtheorem {proposition} {Proposition}

\newtheorem {lemma}  {Lemma}

\newtheorem {remark} {Remark}

\newtheorem {mcor} {Corollary}

\newtheorem{main}{Theorem}

\begin{document}
\renewcommand{\arraystretch}{1.5}

\title[Chaos induced by sliding phenomena in Filippov systems]
{Chaos induced by sliding phenomena\\ in Filippov systems}

\author[Novaes, Ponce \& Var\~ao ]
{Douglas D. Novaes, Gabriel Ponce and R\'egis Var\~{a}o}

\address{ Departamento de Matem\'{a}tica, Universidade
Estadual de Campinas, Rua S\'{e}rgio Baruque de Holanda, 651, Cidade Universit\'{a}ria Zeferino Vaz, 13083--859, Campinas, SP,
Brazil} \email{ddnovaes@ime.unicamp.br} \email{gaponce@ime.unicamp.br} \email{regisvarao@ime.unicamp.br}



\maketitle

\begin{abstract}
In this paper we provide a full topological and ergodic description of the dynamics of Filippov systems nearby a sliding Shilnikov orbit $\Gamma$. More specifically we prove that the first return map, defined nearby $\Gamma$, is topologically conjugate to a Bernoulli shift with infinite topological entropy. In particular, we see that for each $m\in\mathbb{N}$ it has infinitely many periodic points with period $m$.
% \textcolor{red}{, which, in particular, yields the existence of infinitely many $m$-periodic points for the first return map defined nearby $\Gamma,$ for each period $m\in\mathbb{N}$}.
\end{abstract}

%\tableofcontents

\section{Introduction}

Real world problems have been the main motivation on discontinuous differential systems. They are very useful to model phenomena presenting abrupt switches such as electronic relays, mechanical impact, mitosis of living cells, and Neuronal networks. That is one of the reasons for this area to have such a variety of rich examples \cite{BS, BBCK, Co, SJ, physDspecial}.  Therefore the further we understand discontinuous differential systems the more one is prepared to analyse real world problems.
 
When facing a discontinuous differential system, the natural question which rises is ``How to define a consist concept of solutions''. One important paradigm to tackle this question is due to Filippov. In his famous book \cite{F} Filippov studied these systems taking advantage of the well developed theory of differential inclusions \cite{AC}. Then, for a class of discontinuous vector fields $Z$, he provided a branch of rules for what would be a local trajectory of $\dot u=Z(u)$ nearby a point of discontinuity. For instance, consider
\begin{equation}\label{dds}
Z(u)=\left\{\begin{array}{l}
X(u),\quad\textrm{if}\quad h(u)>0,\vspace{0.1cm}\\
Y(u),\quad\textrm{if}\quad h(u)<0,
\end{array}\right.
\end{equation}
where $u\in K,$ being $K$ a closure of an open subset of $\R^n,$  $X,Y$ are $\CC^r$ vector fields, and $h:K\rightarrow\R$ has $0$ as a regular value. The rules stated by Filippov may be applied to establish the notion of local solution of the discontinuous differential system $\dot u=Z(u)$ at a point of discontinuity $\xi\in M=h^{-1}(0).$ Nowadays these rules are known as the Filippov's conventions, and it turns out that for many physical models these conventions are the ones which have physical meaning \cite{BBCK}. Accordingly, discontinuous differential systems rulled by Filippov's conventions are called {\it Filippov systems}. Due to its importance, not only from the mathematical point of view as well as the physical point of view, we shall assume the Fillipov's conventions throughout the paper. Under this convention the switching manifold $M$ can be generically decomposed in  three regions with distinct dynamical behaviours, namely: {\it crossing} $M^c$, {\it sliding} $M^s$, and {\it escaping} $M^e.$ Concisely, the system $\dot u=Z(u)$ may admit solutions either side of the discontinuity $M$ that can be joined continuously, forming a solution that {\it crosses} $M^c\subset M$. Alternatively, solutions might be found to impinge upon $M$, after which they join continuously to solutions that {\it slide} inside $M^{s,e}=M^s\cup M^e\subset M$. See items $(i)$--$(v)$ of section  \ref{pre} for the precise definition of the Filippov's conventions. 

Nonlinear systems may present intricate and complex behaviours such as chaotic motions. Roughly speaking, chaos can  be understood as the existence of an invariant compact set $\Lambda$ of initial conditions for which their trajectories are transitive and exhibit sensitive dependence on $\Lambda$ \cite{W,De,Me}. Each phenomenon from the ordinary theory of differential systems finds its analogous in discontinuous differential systems. However Filippov systems admit a richer variety of behaviours. New chaotic modes rising in discontinuous differential systems have been recently investigated. For instance, in \cite{BCE, CJ} it was studied chaotic set--valued trajectories (nondeterministic chaos), and in \cite{BF1,BF2} the Melnikov ideas was applied to determine the existence of chaos in nonautonomous Filippov systems. Here we shall study deterministic chaos in autonomous Filippov systems.

In this article we analyse 3D Filippov systems admitting a {\it sliding Shilnikov orbit}  $\G$ (see Figure \ref{slidingshil}), which is an entity inherent to Filippov systems. It was first studied in \cite{NT}. In the classical theory a {\it Shilnikov homoclinic orbit} of a smooth vector field is a trajectory connecting a hyperbolic saddle--focus equilibrium to itself, bi--asymptotically.  It is well known that a chaotic behaviour may rise when the Shilnikov homoclinic orbit is unfolded  \cite{PR,Ho}. In the Filippov context pseudo--equilibria are special points contained in $M^{s,e}$ that must be distinguished and treated as typical singularities (see Definition \ref{pequi}). These singularities give rise to the definition of the sliding homoclinic orbit, that is a trajectory, in the Filippov sense, connecting a pseudo--equilibrium  to itself in an infinity time at least by one side, forward or backward.  Particularly a sliding Shilnikov orbit (see Definition \ref{defshil}) is a sliding homoclinic orbit connecting a hyperbolic pseudo saddle–focus $p_0\in M^s$ to it self. This trajectory intersects the boundary $\p M^s$ of $M^s$ at a point $q_0$ (see Figure \ref{slidingshil}).

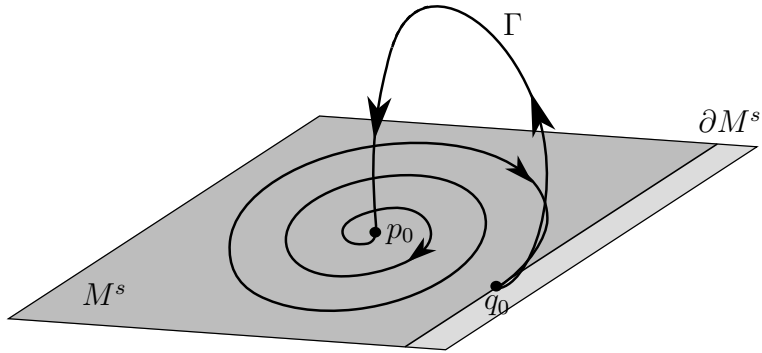
\begin{figure}[h]
\begin{center}
\begin{overpic}[width=10cm]{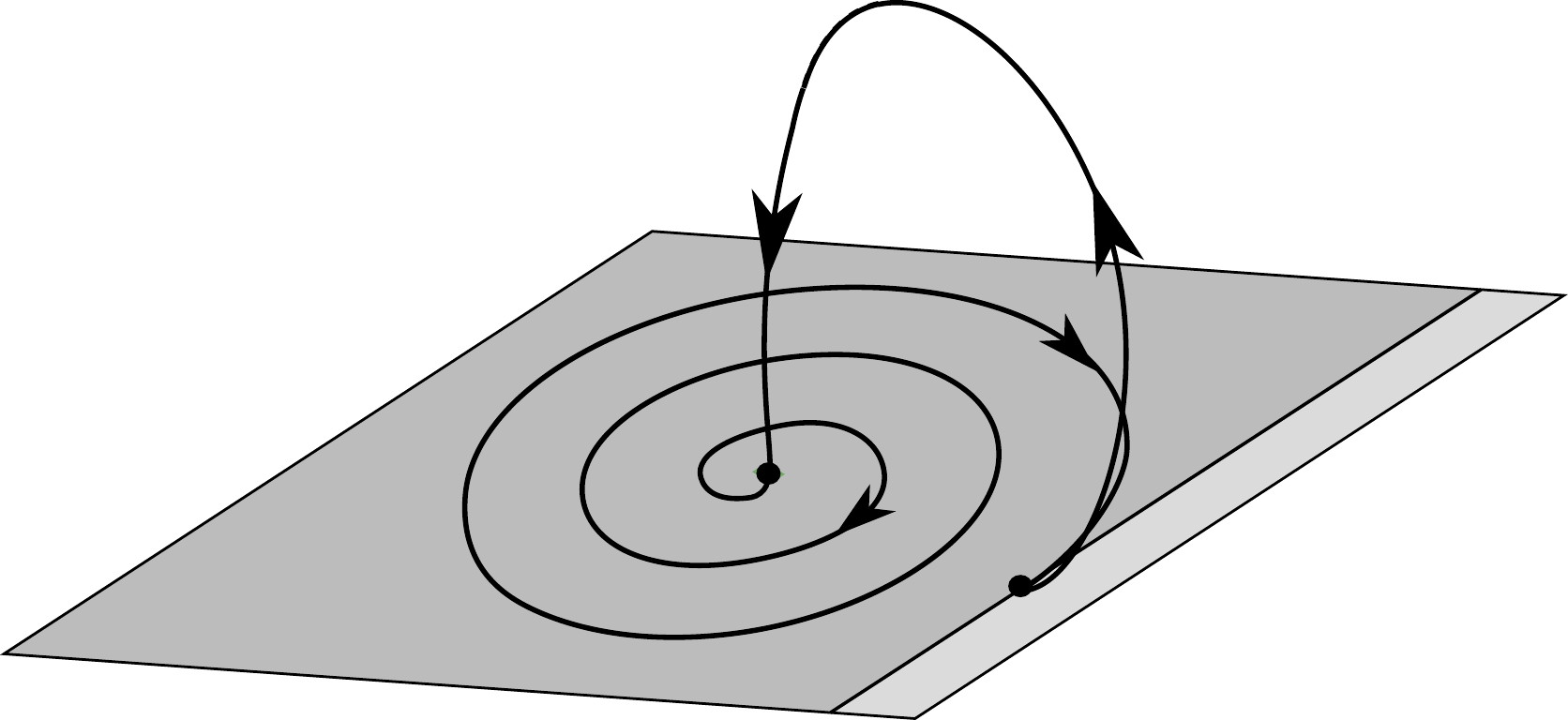}
%\begin{overpic}[grid,tics=5,width=10cm]{shilnikov.pdf}
\put(66,42){$\G$}
\put(50.5,15){$p_0$}
\put(63.5,5.5){$q_0$}
\put(10,6){$M^s$}
\put(92,29){$\p M^s$}
\end{overpic}
\end{center}

\bigskip

\caption{The point $p_0\in M^s$ is a hyperbolic pseudo saddle–focus. The trajectory $\G$, called Shilnikov sliding orbit, connects $p_0$ to itself passing through the point $q_0\in\p M^s$. We note that the flow leaving $q_0$ reaches the point $p_0$ in a finite positive time, and approaches backward to $p_0$, asymptotically. }\label{slidingshil}
\end{figure}

In dynamics one is concerned to get the most possible complexity from a given system, that is why often for systems which exhibit some complexity one is able to find a Bernoulli shift as a factor. For smooth dynamical systems one may benefit from certain geometrical structures (hyperbolicity) which imply the existence of stable and unstable manifolds. That has been the case in many studies, we mention two classical works done by Bowen \cite{bowen} and Tresser \cite{Tre}. The systems we shall study in this paper do not benefit from these geometrical structures, so we have to use the properties of the Filippov systems themselves to overcome this lacking of structures (e.g. the well-posedness of the $\eta_*$ function, see \S \ref{sec:proof}).

Let $Z$ be a 3D discontinuous vector field like \eqref{dds} defined on $K\subset \R^3$. Assume that the Filippov system $\dot u=Z(u)$ admits a sliding Shilnikov orbit $\G$ connecting a hyperbolic pseudo saddle–focus $p_0\in M^s$ to it self and intersecting the curve $\p M^s$ at a point $q_0$. Consider a small neighbourhood $\mathcal U$ of $q_0$ in $\p M^s$, and denote by $\Lambda \subset \mathcal U$ the points which return infinitely often by the forward orbit of the flow to this neighbourhood of $q_0$, or, if you wish, the set $\Lambda$ is the maximal set in this neighbourhood which is $\pi$ invariant (see subsection \ref{1streturn} for more information). As we shall see, this set is non vanishing. The complexity of a flow is interpreted as the complexity of its returning map $\pi|\Lambda$. In this context our main result states that $\pi$ can be as much chaotic as one wishes:

\begin{main}\label{theo:topological}
There is a small neighbourhood $\mathcal U \subset \partial M^s$ of $q_0$ and a set $\Lambda \subset \mathcal U$ as above such that: 

\begin{itemize}
 \item[a)] for each $k \in \mathbb N$ there exists a $\pi$-invariant cantor set $\Lambda_k \subset \Lambda$ such that $\pi|\Lambda_k$ is conjugate to the shift on $\Sigma_2 \times \Sigma_k^*$, that is $$h_k\circ \sigma_k  = \pi \circ h_k $$
where $h_k: \Sigma_2 \times \Sigma_k^*  \rightarrow \Lambda_k$ is a homeomorphism. In particular the dynamics on $\Lambda_k$ is transitive, sensitive to initial conditions and have dense periodic points.

\item[b)] There is a homeomorphism $h: \Sigma_2 \times \Sigma^b \rightarrow \Lambda := \bigcup_{k}\Lambda_k$ such that $h$ conjugates the dynamics of $\sigma$ and $\pi$ and $\Lambda \cup \{ q_0 \}$ is a compact set. In particular the topological entropy of $\pi$ is infinite. 
\end{itemize}
\end{main}
For the definitions of $\Sigma_2,$ $\Sigma_k^*,$ $\Sigma^b,$ $\sigma_k,$ and $\sigma$ see $\eqref{sigmak},$ $\eqref{sigmak*},$ $\eqref{sigmab},$ $\eqref{sigmakk},$ $\eqref{sigma},$ respectively. 

We remark that in view of Theorem \ref{theo:topological}, given any natural number $m\geq 1$ we can find infinitely many periodic points for the first return map with period $m$ and, consequently, infinity many closed orbits of $\dot u=Z(u)$ nearby $\G$. Indeed, given $k\geq 1$, each periodic point of period $m$ for $\sigma_k$ is mapped by $h_k$ in a periodic point of period $m$ for $\pi$, thus varying $k \geq 1$ we obtain infinitely many periodic points of a fixed period $m$ for $\pi$.

We also obtain another two consequences from Theorem \ref{theo:topological}. The first one states that we are able to understand any compact invariant set for the first returning map (or flow) by some dynamics on a symbolic dynamics (Corollary \ref{cor:compact.invariant}). Another consequence is that this system is also as chaotic as one can be from the ergodic point of view (Corollary \ref{theo:ergodic}).

\section{Filippov system and the sliding Shilnikov orbit}\label{pre}

In this section we shall introduce the basic concepts of Filippov systems as well the definition of {\it sliding Shilnikov orbit}. We remark that a major part of this section is constituted by a well known theory and may be found in other works (see for instance \cite{F,GST,NT}).

Let $K$ be the closure of an open subset of $\R^n$, and let $X,Y\in \CC^r(K,\R^3)$, be  $\CC^ r$ vector fields defined on $K\subset\R^3$. We denote by $\Omega_h^r(K,\R^3)$ the space of piecewise vector fields 
\begin{equation}\label{omega}
Z(u)=\left\{\begin{array}{l}
X(u),\quad\textrm{if}\quad h(u)>0,\vspace{0.1cm}\\
Y(u),\quad\textrm{if}\quad h(u)<0,
\end{array}\right.
\end{equation}
defined on $K$, being $0$ a regular value of the differentiable function $h:K\rightarrow\R$. As usual, system \eqref{omega} is denoted by $Z=(X,Y)$ and the surface of discontinuity $h^{-1}(0)$ by $M$ . So $\Omega_h^r(K,\R^3)=\CC^r(K,\R^3)\times \CC^r(K,\R^3)$ is endowed with the product topology, while $\CC^r(K,\R^3)$ is endowed with the $\CC^r$ topology. We concisely denote $\Omega_h^r(K,\R^3)$ and $\CC^r(K,\R^3)$ only by $\Omega^r$ and  $\CC^r$, respectively. 

In order to understand the Filippov's convention for the discontinuous differential system $\dot u=Z(u)$  we need to distinguish some regions on $M$. The points on $ M$ where both vectors fields $X$ and $Y$ simultaneously point outward or inward from $M$ define, respectively, the {\it escaping} $M^e$ or {\it sliding} $M^s$ regions, and the interior of its complement in $M$ defines the {\it crossing region} $M^c$. The complementary of the union of those regions  constitute by the {\it tangency} points between $X$ or $Y$ with $M$. 

The points in $M^c$ satisfy $Xh(\xi)\cdot Yh(\xi) > 0$, where $Xh(\xi)=\langle \nabla h(\xi), X(\xi)\rangle$. The points in $ M^s$ (resp. $ M^e$) satisfy $Xh(\xi)<0$ and $Yh(\xi) > 0$ (resp. $Xh(\xi)>0$ and $Yh(\xi) < 0$). Finally, the tangency points of $X$ (resp. $Y$) satisfy $Xh(\xi)=0$ (resp. $Yh(\xi)=0$).

Now we define the {\it sliding vector field}
\begin{equation}\label{slisys}
\widetilde Z(\xi)=\dfrac{Y h(\xi) X(\xi)-X h(\xi) Y(\xi)}{Y h(\xi)- Xh(\xi)}.
\end{equation}

\begin{definition}\label{pequi} A point $\xi^*\in M^{s,e}$ is called a {\it pseudo--equilibrium} of $Z$ if it is a singularity of the sliding vector field, i.e. $\widetilde Z(\xi^*)=0$. When $\xi^*$ is a hyperbolic singularity of $\widetilde Z$, it is called a {\it hyperbolic pseudo--equilibrium}. Particularly if $\xi^*\in M^s$ (resp. $\xi^*\in M^e$) is an unstable (resp. stable) hyperbolic focus of $\widetilde Z$
then we call $\xi^*$ a {\it hyperbolic pseudo saddle--focus}.
\end{definition}

Let $\f_W$ denotes the flow of a smooth vector field $W$. 
The local trajectory $\f_Z(t,p)$ of $\dot u= Z(u)$ passing through a point $p\in\R^3$ is given by the Filippov convention (see \cite{F,GST}). Here $0\in I_p\subset\R$ denotes the maximum interval of definition of $\f_Z(t,p)$, and $\f_W$ denotes the flow of a vector field $W$. The Filippov convention is summarized as following: 

\begin{itemize}
\item[$(i)$] for $p\in\R^3$ such that $h(p)>0$ (resp. $h(p)<0$) and taking the origin of time at $p$, the trajectory is defined as $\f_Z(t,p)=\f_X(t,p)$ (resp. $\f_Z(t,p)=\f_Y(t,p)$) for $t\in I_p$.

\item[$(ii)$] for $p\in M^c$ such that $(Xh)(p),(Yh)(p)>0$ and taking the origin of time at $p$, the trajectory is defined as $\f_Z(t,p)=\f_Y(t,p)$ for $t\in I_p\cap\{t<0\}$ and $\f_Z(t,p)=\f_X(t,p)$ for $t\in I_p\cap\{t>0\}$. For the case $(Xh)(p),(Yh)(p)<0$ the definition is the same reversing time;

\item[$(iii)$] for $p\in M^s$ and taking the origin of time at $p$, the trajectory is defined as $\f_Z(t,p)=\f_{\widetilde{Z}}(t,p)$ for $t\in I_p\cap\{t\geq 0\}$ and $\f_Z(t,p)$ is either $\f_X(t,p)$ or $\f_Y(t,p)$ or $\f_{\widetilde{Z}}(t,p)$ for $t\in I_p\cap\{t\leq 0\}$. For the case $p\in M^e$ the definition is the same reversing time;

\item[$(iv)$]  For $p\in\p M^c\cup\p M^s\cup\p M^e$ such that the definitions of trajectories for points in $ M$ in both sides of
$p$ can be extended to $p$ and coincide, the orbit through $p$ is this limiting orbit. We will call these
points {\it regular tangency} points.

\item[$(v)$] for any other point ({\it singular tangency} points) $\f_Z(t,p)=p$ for all $t\in\R$;
\end{itemize}

A tangency point $p\in M$ is called a {\it visible fold} of 
$X$ $($resp. $Y)$ if $X^2h(p)>0$ $($resp. $Y^2h(p)<0)$.  Analogously, reversing the inequalities, we define an {\it invisible fold}. A fold $p$ of $X$ (resp. $Y$), visible or invisible, such that $Yh(p)\neq 0$ (resp. $Xh(p)\neq0)$ is called a {\it regular--fold} point. The regular--fold points are examples of $(iii),$ regular tangency points.

\begin{remark}For a visible regular-fold $p$ such that $Yh(p)>0$, taking the origin of time at $p$, the trajectory passing through $p$ is defined as $\f_Z(t,p)=\f_1(t,p)$ for $t\in I_p\cap\{t\leq 0\}$ and $\f_Z(t,p)=\f_2(t,p)$ for $t\in I_{p}\cap\{t\geq 0\}$, where each $\f_1,\f_2$ is either $\f_X$ or $\f_Y$ or $\f_{\widetilde{Z}}$. 
\end{remark}

The next result provides the dynamics of the sliding vector field near regular-fold points. A proof of that can be find in \cite{T}.
\begin{proposition}[\cite{T}]\label{transv}
Given $Z=(X,Y)\in\Omega^r$ if $p\in\p M^{e,s}$ is a fold--regular point of $Z$ then the sliding vector field $\widetilde Z$ \eqref{slisys} is transverse to $\p M$ at $p$.
\end{proposition}

In \cite{NT} it has been introduced the concept of {\it sliding Shilnikov orbits}, and some of their properties were studied. In what follows we give the definition of this object and a result about periodic orbits near it.
\begin{definition}[Sliding Shilnikov orbit]\label{defshil}
Let $Z=(X,Y)$ be a 3D discontinuous vector field having a hyperbolic pseudo saddle--focus $p_0\in  M^{s}$ $($resp. $p_0\in  M^{e})$. We assume that there exists a tangential point $q_0\in\p M^s$ $($resp. $q_0\in\p M^e)$ which is a visible fold point of the vector field $X$ such that
\begin{itemize}
\item[$(j)$] the orbit passing through $q_0$ following the sliding vector field $\widetilde Z$ converges to $p_0$ backward in time $($resp. forward in time$)$;

\item[$(jj)$] the orbit starting at $q_0$ and following the vector field $X$ spends a time $t_0>0$ $($resp. $t_0<0)$ to reach $p_0$.
\end{itemize}
So through $p_0$ and $q_0$ a sliding loop $\G$ is characterized. We call $\G$ a {\it sliding Shilnikov orbit} $($see Figure \ref{slidingshil}$)$.  Accordingly we denote $\G^+=\G\cap\{u\in K:\, h(u)>0\}$ and $\G^s=\G\cap M^s$.
\end{definition}

\begin{theorem}[\cite{NT}]\label{t1}
Assume that $Z=(X,Y)\in\Omega^r$ $($with $r\geq0)$ has a sliding Shilnikov orbit $\G$. Then every neighbourhood $G\subset\R^3$ of $\G$ contains infinitely many sliding periodic orbits of $Z$;
\end{theorem}

When $p_0\in M^s$ and $q_0\in\p M^s$ the solutions of $\dot u=Z(u)$ are uniquely defined forward time in a small neighbourhood of $\G$, while for $p_0\in M^e$ and $q_0\in\p M^e$ the solutions of $\dot u=Z(u)$ are uniquely defined backward time but not forward. So throughout this paper we shall concentrate our attention in the case $p_0\in M^s$ and $q_0\in\p M^s$.

In what follows we present an example of a $2$--parameter family of discontinuous piecewise linear differential system $Z_{\al,\beta}$ admitting, for every positive real numbers $\al$ and $\beta$, a sliding Shilnikov orbit $\G_{\al,\be}$. This family was studied in \cite{NT}.

\begin{equation}\label{s1}
Z_{\al,\beta}(x,y,z)=\left\{\begin{array}{ll}
X_{\al,\beta}(x,y,z)=\left(
\begin{array}{c}
-\al\\
x-\beta\\
y-\dfrac{3\beta^2}{8\al}
\end{array}
\right)&\textrm{if}\quad z>0,\vspace{0.2cm}\\

Y_{\al,\beta}(x,y,z)=\left(
\begin{array}{c}
\al\\
\dfrac{3\al}{\beta}y+\beta\\
\dfrac{3\beta^2}{8\al}
\end{array}
\right)&\textrm{if}\quad z<0.
\end{array}\right.
\end{equation}
The switching manifold for system \eqref{s1} is given $\Sigma=\{z=0\}$, and decomposed as $\Sigma=\ov{\Sigma^c}\cup\ov{\Sigma^s}\cup\ov{\Sigma^e}$ being
\[
\begin{array}{llll}
\Sigma^c=\left\{(x,y,0):\,y>\dfrac{3\be^2}{8\al}\right\}, &\Sigma^s=\left\{(x,y,0):\,y<\dfrac{3\be^2}{8\al}\right\}&\textrm{and}&\Sigma^e=\emptyset.
\end{array}
\]
The origin $p_0=(0,0,0)$ is a hyperbolic pseudo saddle--focus of system $Z_{\al,\beta}$ \eqref{s1} in such way that its projection onto $\Sigma$ is an unstable hyperbolic focus of the sliding vector field $\widetilde Z_{\al,\beta}$ \eqref{slisys} associated with \eqref{s1}. Moreover there exists a sliding Shilnikov orbit $\G_{\al,\beta}$ connecting $p_0=(0,0,0)$ to itself and passing through the fold--regular point $q_0=\big(3\be/2,3\be^2/(8\al)\big)$.

\subsection{The first return map}\label{1streturn} The behaviour of a system close to a sliding Shilnikov orbit can be understood by studying the first return map in a small neighbourhood $I\subset\p M^s$ of $q_0$. In what follows we shall define this map.

For $\xi\in M^s$ and $z\in\R^3$, let the functions $\f_s(t,\xi)$ and $\f_X(t,z)$ denote the solutions of the differential systems induced by $\widetilde Z$ and $X$, respectively, such that $\f_s(0,\xi)=\xi$ and $\f_X(0,z)=z$. 

Take $\gamma_r=\ov{B_r(q_0)\cap\p M^s}$. Here $B_r(q_0)\subset  M$ is the planar ball with center at $q_0$ and radius $r$. Of course $\gamma_r$ is a branch of the fold line contained in the boundary of the sliding region $\p M^s$. From definition \ref{defshil}, $\f_X(t_0,q_0)=p_0\in M^s$, moreover, the intersection between $\Gamma^+$ and $M$ at $p_0$ is transversal. So taking $r>0$ sufficiently small, we find a function $\tau(\xi)>0$, defined for $\xi\in\gamma_r$, such that  $\tau(q_0)=t_0$ and $\f_s(\tau(\xi),\xi)\in M^s$ for every $\xi\in\gamma_r$. 

The forward saturation of $\gamma_r$ (resp. $\p M^s,$ $\gamma_r^+,$ and $\gamma_r^-$) through the flow of $X$ meets $M$ in a curve $\mu_r$, that is $\mu_r=\{\f_s(\tau(\xi),\xi):\,\xi\in\gamma_r\}$. So let $\T:\gamma_r\rightarrow \mu_r$ denote the bijection $\T(\xi)=\f_s(\tau(\xi),\xi)$.

Now, Proposition \ref{transv} implies that the intersection between $\G^s$ and $\p M^s$ at $q_0$ is transversal. So in addition, taking $r>0$ small enough, the backward saturation $S_r$ of $\gamma_r$ through the flow of $\widetilde Z$ converges to $p_0$. Therefore
\[
S_r\cap\mu_r=\bigcup_{i=1}^{\infty} J_i,
\]
where $J_i\cap J_j=\emptyset$ if $i\neq j$ and $J_i\to \{p_0\}$. For each $i=1,2,\ldots$, we take $I_i=\T^{-1}(J_j)$. Clearly $I_i\cap I_j=\emptyset$ if $i\neq j$ and $I_i\to\{q_0\}$. Set 
\begin{equation}\label{CI}
\CU_r=\bigcup_{i=1}^{\infty} I_i.
\end{equation}
Therefore the first return map $\pi:\CU_r\rightarrow \CU_r$ is well defined. Moreover it can be taken as
\[
\pi(\xi)=\f_s\big(\tau_s(\f_X(\tau(\xi),\xi)),\f_X(\tau(\xi),\xi)\big),
\]
where, for each $\xi'\in\mu\setminus\{p_0\}$,  $\tau_s(\xi)>0$ denote the time such that $\f_s(\tau_s(\xi'),\xi')\in\g_r$. 
%The existence of such times is guaranteed by the transversality between $\G$ and $M$, and between the flow of $\widetilde Z$ and the fold line $\p M^s$ (see Proposition \ref{transv}), respectively. 
% 
% 
% 
%
%
%
%

The next result estimates the derivative of the first return map.
\begin{proposition}\label{prop:pi.expands}
Consider $\CU_r$ as defined in \eqref{CI}. There exists $r>0$ sufficiently small such that $|\pi'(\xi)|>1$ for every $\xi\in \CU_r$.
\end{proposition}
%\begin{proposition}\label{prop:pi.expands}
%There is a small neighbourhood $\mathcal{U}\subset \p M^s$ of $q_0$ such that $|\pi'(\xi)|>1$ for every $\xi\in \mathcal{U}\setminus\{q_0\}$.
%\end{proposition}
\begin{proof}

Denote by $\g^{+}_r\subset \g\setminus\{q_0\}$ (resp. $\g^{-}_r\setminus\{q\}$) the segment on the right (resp. left) of $q_0$. Accordingly denote $\mu_r^ {\pm}=\T(\g_r^{\pm})$.

For each $R>0$, the focus $p_0\in M^s$ of the sliding vector field $\widetilde Z$ is contained in $\mu_{R}$ which is traversal to the flow of $\widetilde Z$. So there exists $R_0>0$ such that for every $0<R\leq R_0$  it  is well defined a first return map on $\mu_{R}^{+}$ (resp. $\mu_{R}^{-}$), which we denote by $\rho_{\pm}:\mu_{R}^{\pm}\rightarrow \mu_{R}^{\pm}$. Since $p_0$ is a hyperbolic unstable fixed point of $\rho_{\pm}$, there exists a neighbourhood $U\subset \mu_{R_0}$ and homomorphisms $H_{\pm}:U\rightarrow U$ such that $\rho_{\pm}(\zeta)=H_{\pm}(\la H_{\pm}^{-1}(\zeta))$, with $|\la|>1$, and then $\rho_{\pm}^k(\zeta)=H_{\pm}(\la^k H_{\pm}^{-1}(\zeta))$. The radius $R$ can be taken small enough in order to ensure that $\mu_{R}^{\pm}\subset U$.

The backward saturation of $\g_{R}$ through the flow of $\widetilde Z$ intersects $\mu_{R}$ many times, indeed it converges to $p_0$. So denote by $S\subset \mu_{R}$ the first connected component of this intersection which is entirely contained in $U$. Clearly either $S\subset \mu^+$ or $S\subset\mu^-$. Without loss of generality we assume that $S\subset \mu^+$. The flow of $\widetilde Z$ induces a diffeomorphism $\widetilde \rho$ between $S$ and $\g_{R}$. Moreover the flow of $X$ induces a diffeomorphism $\rho_X$ between $\g_R$ and $\mu_{R}$.

Since $\widetilde\rho$ and $\rho_X$ are diffeomorphism, there exists $\widetilde{\al}>0$ and $\al_X>0$ such that $\widetilde{\al}=\min\{|\widetilde{\rho}\,'(\zeta)|:\,\zeta\in S\}$ and $\al_X=\min\{|\rho_X'(\xi)|:\,\xi\in \g_R\}$.

Now given $k_0\in\N$, there exists $0<r<R$ such that $\rho_+^{k}(\mu_r)\cap S=\emptyset$ for $k\leq k_0$. Particularly, we take $k_0$ such that $\al_X\widetilde{\al}|\la|^{k_0}>1$.

For $\xi\in\mathcal{U}_r$, $\rho_X(\xi)\in \mu_r,$ denote $\ov k=k(\rho_X(\xi))$. Clearly there exists a neighborhood $W\subset \mathcal{U}_r$ of $\xi$ such that $k(\rho_X(w))=\ov k$ for every $w\in W$. Therefore the first return map reads 
\[
\begin{array}{rl}
\pi(\xi)&=\widetilde\rho\circ\rho_+^{\ov k}\circ\rho_X(\xi)\\
&=\widetilde\rho\circ H_{+}(\la^{\ov k} H_{+}^{-1}(\rho_X(\xi))).
\end{array}
\]
Hence $|\pi'(\xi)|\geq \al_X\widetilde{\al}|\la|^{\ov k}>1,$ because $\ov k>k_0$. It concludes this proof.
\end{proof}

\section{Basic facts on Bernoulli Shifts}

Our main result (Theorem \ref{theo:topological}) completely characterizes, from the topological point of view, the dynamical around the homoclinic orbit. As a consequence we are able, on Corollary \ref{theo:ergodic}, to give a complete characterization from the ergodic point of view as well. Therefore in order to one fully understand this ergodic point of view we describe some basic facts from the Ergodic Theory. The reader may find a good introduction to the subject on \cite{KH} and the references therein.

Let $(X,\mathcal A,\mu)$ be a probability space and $f:X \rightarrow X$ be a measurable function. We say that a measurable set $B \subset X$ is $f$-invariant if
\[f^{-1}(B)=B \operatorname{mod} 0,\]
where $\operatorname{mod} 0 $ means that except a measure zero set both sets are equal.
We say that $f$ preserves the measure $\mu$, or that $f$ is $\mu$-invariant, when 
\[\mu(f^{-1}(B)) = \mu(B)\]
for every measurable set $B \subset X$. \\

Given a measurable preserving function $f:X \rightarrow X$ in a probability space $(X,\mu)$, we say that $f$ is ergodic if, and only if, for every $f$-invariant measurable set $B\subset X$ we have
\[\mu(B) = 0 \text{ or } \mu(B)=1.\]

Ergodicity is a very important property in Dynamical Systems and it roughly means that the dynamics cannot be broken in smaller simple dynamics. Hence it actually implies a certain type of chaos for a system with respect to a given measure. Bernoulli shifts have a fairly simple description and still amazingly they are the most chaotic possible examples. We now describe the Bernoulli shifts.

Given any natural number $k \in \mathbb N$, we define the space of all sequences of natural numbers between $0$ and $k-1$ by
\begin{equation}\label{sigmak}
\Sigma_k = \{0,1,...,k-1 \}^{\mathbb N}.
\end{equation}
Due to a more intuitive approach in the proof of Theorem \ref{theo:topological} we will need the set 
%it  and along the notes we will denote the space of all sequences of natural numbers between $1$ and $k$ by  \textcolor{red}{Tem que começar do 1, certo?}
\begin{equation}\label{sigmak*}
\Sigma^{*}_k = \{1,...,k \}^{\mathbb N},
\end{equation}
which we point out it is not a standard notation on symbolic dynamics.

These are a countable product space where each coordinate is a discrete compact space. By Tychonoff's theorem $\Sigma_k$ (respectively $\Sigma^*_k$) is compact with the product topology induced by the discrete topology of $\{0,1,...,k-1\}$  (respectively $\{1,...,k\}$). A metric in this space which generates the product topology is described by the distance

We define the distance on $\Sigma_k$ by: 
\begin{eqnarray*}
 d:\Sigma_k \times \Sigma_k &\rightarrow& \mathbb R\\
 d(\alpha, \beta) &=& \left\{ \begin{array}{c}
                               0 \; \text{ se } \alpha=\beta\\
                               \left(\frac{1}{2}\right)^n ,\; n=max\{ a \in \mathbb N: \alpha(i)=\beta(i), |i|\leq a \}
                              \end{array}
 \right.
\end{eqnarray*}

Acting on $\Sigma_k$ we have the so called one-sided Bernoulli shift $\sigma: \Sigma_k \rightarrow \Sigma_k$, which simply operates a left-translation on each sequence, that is, given any sequence $(x_n)_{n \in \mathbb N}$ the image if this sequence is the sequence
\[\sigma((x_n)_{n \in \mathbb N}) = (x_{n+1})_{n \in \mathbb N}.\]

\begin{definition}
Given $n\in \mathbb N$ and $m$ values $a_1,a_2,..,a_m \in \{ 0,1,...,k-1\}$. We denote by $C(n; a_1,a_2,...,a_m)$ the set defined by
\[C(n; a_1,a_2,...,a_m) = \{ (x_i)_{i\in \mathbb N} : x_{n+1} = a_1, x_{n+2} = a_2,...,x_{n+m}=a_m\}.\]
The sets of this form are called cylinders. 
\end{definition}

Let $\mathfrak C$ be the family of all cylinders in $\Sigma_k$. This family generates a $\sigma$-algebra $\mathcal C$, which will be the standard $\sigma$-algebra to work with on $\Sigma_k$. 

To define a measure on $\Sigma_k$ we take any probability vector $p=(p_0,...,p_{k-1})$ (i.e. $p_i \in [0,1]$ and $\sum_i p_i = 1$). The probability vector $p$ defines, in a trivial way, a measure $p$ on $\{0,1,...,k-1\}$. Thus, we can take $\mu$ as the product measure $\mu = p^{\mathbb N}$ on $\Sigma_k$. This measure is characterized by its values on cylinders. Given a cylinder  $C(n; a_1,a_2,...,a_m)$, one can easily see that
\[\mu(C(n; a_1,a_2,...,a_m)) = p_{a_1}\cdot p_{a_2} \cdot ... \cdot p_{a_m}.\]

The measure $\mu$ is called a \textit{Bernoulli measure}. It is easy to see that $\mu$ is $\sigma$ invariant for any Bernoulli shift $\sigma:\Sigma_k \rightarrow \Sigma_k$. Also, the system $(\sigma,\mu)$ is ergodic. A measurable automorphism $f:X \rightarrow X$ of a probability space $(X,\mu)$ is called a \textit{Bernoulli automorphism} if it is isomorphic to a Bernoulli shift $\sigma:\Sigma_k \rightarrow \Sigma_k$ for some $k\in \mathbb N$. By an isomorphism we mean a bimeasurable function that conjugates the dynamics and takes $\mu$ to a Bernoulli measure. That is a Bernoulli automorphism preserves all ergodic properties of a Bernoulli shift.

The following proposition is a very well-known fact from the theory of Bernoulli shifts (e.g. \cite{KH}).

\begin{proposition} \label{prop:periodos}
For any $k \in \mathbb N$, the Bernoulli shift $\sigma:\Sigma_k \rightarrow \Sigma_k$ has periodic orbits of all periods and the set of transitive points is a residual set.
\end{proposition}

Along the paper we will work with the spaces $\Sigma_2 \times \Sigma_k^{*}$. For each $k \in \mathbb N^{*}$, on each $\Sigma_2 \times \Sigma_k^{*}$ we have the shift on two-coordinates
\begin{equation}\label{sigmakk}
\sigma ((x_n)_n ,(y_m)_m) = ((x_{n+1})_n, (y_{m+1})_m).
\end{equation}

This shift on two coordinates is isomorphic to a standard shift on $\Sigma_{2\cdot k}$, so it is also a Bernoulli automorphism. It is also a direct fact that the two-coordinates shift above is topologically conjugate to the standard shift on $\Sigma_{2\cdot k}$, thus the conclusions of Proposition \ref{prop:periodos} are also true for the two-coordinates shift.

Let us define 
\begin{equation}\label{sigmab}
\Sigma^{b} := \bigcup_{k \in \mathbb N}\Sigma^{*}_k=\{ \{ x_i\}_i | \exists L \in \mathbb R \text{ s.t. }|x_i|\leq L \; \forall i \}.
\end{equation}
We consider the two-coordinates shift
\begin{equation}\label{sigma}
\sigma : \Sigma_2 \times \Sigma^{b} \rightarrow \Sigma_2 \times  \Sigma^{b}.
\end{equation}

To make notations easier, we will denote by $\sigma_k$ the restriction of $\sigma$ to the space $ \Sigma_2 \times \Sigma^{*}_k$. Hence $\sigma_k: \Sigma_2 \times \Sigma^{*}_k \rightarrow \Sigma_2 \times \Sigma^{*}_k.$

The space of sequences $\Sigma_2 \times \Sigma^{*}_k$ is naturally endowed with the product topology, which is the coarsest topology for which the cylinders are open set. And the topology on $\Sigma_2 \times \Sigma^b$ is the coarset topology having the cylinders of $\Sigma_2 \times \Sigma^{*}_k$ as open sets $\forall k \in \mathbb N$. We note that $\sigma_k$ is defined on a compact space while $\sigma$ is not.

One of the most useful invariants on Dynamical Systems and Ergodic Theory is the topological entropy. One may think of topological entropy as a measurement of chaos. Topological entropy has a not so straight definition (which we recomend the reader to take a look \cite{KH}) but fortunately to our context it can be associated to the growth of periodic points. Therefore to purpose we consider the topological entropy as follows. For a compact $\pi$-invariant set $\Omega \subset \Sigma_2 \times \Sigma^b$ we define the topological entropy of $\pi|\Omega$ as
$$h_{\sigma|\Omega}:= \lim_{n \rightarrow \infty}\frac{1}{n} \# Per_n(\sigma|\Omega),$$
 where $\#Per_n(f)$ means the number of periodic point of period $n$. It is not difficult to prove that $h_{\sigma_k}=log(2k)$.
 
\section{Proof of Theorem \ref{theo:topological}}\label{sec:proof}

Consider the Filippov system $\dot u=Z(u)=(X,Y)(u)$ given by \eqref{omega}, and denote by $\phi(t,v)$ its (Filippov) solution such that $\phi(0,v)=v$. Assume that $Z$ contains a sliding Shilnikov orbit $\G,$ and let $p_0\in M^s$ and $q_0\in\p M^s$ be as in Definition \ref{defshil}. We consider a neighbourhood $I\subset \p M^s$ of $q_0$ for which the first return map $\pi$ is well defined. We assume that $I$ has end points $q_1$ and $q_2$ and we denote $I=[q_1,q_2]$. The forward saturation of $I$ through the flow of $X$ intersects $M$ in a curve $J.$ 

Let us call by $\Lambda$ the set of points in $I$ which return infinitely often to $I$ through the forward flow of $Z$. That is 
$$\Lambda = \{\xi\in I \; | \; \exists \{t_n\}_{n \in \mathbb N}, t_n \rightarrow \infty, \phi(t_n,\xi)\in I \}.$$
We note that Theorem \ref{t1} guarantees that $\Lambda\neq\emptyset$.

Call $I_0 := [q_1,q_0]$, $I_1 := [q_0,q_2]$, $J_0$ and $J_1$ denote the intersection of $M$ with the forward saturation of $I_0$ and $I_1$, respectively. Given a point $ \xi \in \Lambda$ we denote by $\eta_*(\xi)$, $* \in \{0,1 \}$ the number of intersections that the forward flow orbit of $\xi$ has with $J_*$ before returning to $\Lambda \subset I$, that is,
\[\eta_*(\xi) := \# \{\varphi(t,\xi) \cap J_* : 0<t<t_{\xi}    \} .\]
where $t_\xi$ is the first return time of $\xi$ on $\Lambda$. The intersections detected by the function $\eta_*(\xi)$ are occurring on the sliding region $M^s$ of the switching surface $M$. That means, for a point $\xi\in\p M^s$ sufficiently close to $q_0$, the flow starting at $\xi$ travels forward in time following the vector field $X$. After a finite time it reaches transversally the switching surface at a point of $J\subset M^s$ close to $p_0$. Then the flow follows the sliding vector field $\widetilde Z$  (see \eqref{slisys}), spiralling outward around $p_0$ until reaching the fold line $\p M^s$. Since the curve $J$ is transversal to the sliding vector field $\widetilde Z$  and contains the pseudo saddle--focus $p_0$, the number $\eta_*(\xi)$ is well defined. 
Notice that $\eta_*(\xi)$ counts the amount of times that the flow of $\widetilde Z$ intersects $J_*$, it could, of course, turn around $p_0$ several times more before reaching $\Lambda$ without intersect $J$.

We will construct a map $$h_k:\Sigma_2 \times \Sigma_{k}^* \rightarrow \Lambda$$ that will conjugate the dynamics of $\sigma_k$ with $\pi$ (i.e. $ h_k \circ \sigma_k = \pi \circ h_k$), where $\Sigma_k^* = \{ 1,2, \ldots, k \}^{\mathbb N}$.

Fix a natural number $k>0$ and take a point $$(X,N) = ((x_i)_{i\in \mathbb N} , (n_i)_{i\in \mathbb N}) \in \Sigma_2 \times \Sigma_k^* $$ We will define $h_k((X,N))$ through a limit process.

Define $P_0(X,N)$ as the points which are in $I_{x_0}$, that is  $P_0(X,N) = I_{x_0}$. Define $P_1(X,N)$ as the points which are in $I_{x_0}$ and before arriving by the first return maps to $I_{x_1}$ touches $n_0$ times the segment $J_{x_1}$, that is:
$$P_1(X,N) = \{ \xi \in P_0(X,N)\; | \; \eta_{x_1}(\xi)=n_0, \pi(\xi)\in I_{x_1} \}.$$
In general we define
$$ P_{m+1}(X,N) = \{ \xi \in P_m(X,N) | \eta_{x_{m+1}}(\pi^m(\xi))=n_m, \; \pi^{m+1}(\xi) \in I_{x_{m+1}}\}.$$

Now consider the following set
\begin{equation}\label{P}
P(X,N) : = \bigcap_{i\in \mathbb N} P_i(X,N).
\end{equation}
Notice that $P_i(X,N) \subset P_{i-1}(X,N)$ and each $P_i(X,N)$ is a closed interval. Hence $P{(X,N)}$ is a point or a non-degenerated interval, we want to rule out the non-degenerate interval case. We start with the following. 

\begin{lemma}\label{l1}
 If $P(X,N)\cap P(X', N') \neq \emptyset$, then $P(X,N)= P(X', N')$. In particular $(X,N)=(X',N')$.
\end{lemma}
\begin{proof}
 That comes directly from the definition of the sets, because $P(X,N)$ is solely defined by stating what the orbit of a point "behaves", hence if a point is also on $P(X',N')$ that means the sets are the same.
\end{proof}

\begin{lemma}
  $\pi(P(X,N))$ is of the form $P(X',N')$
\end{lemma}
\begin{proof}
 In fact one have $\pi(P(X,N)) = P(\sigma(X,N))$, this comes once again from the definition of the set.
\end{proof}

\begin{lemma}
 If $\pi: \Lambda \rightarrow \Lambda$ is such that $|\pi'(\xi)|>1$ for all $\xi \in \Lambda$, then $P(X,N)$ is a point $\forall (X,N) \in \Sigma_2 \times \Sigma^*_k$.
\end{lemma}
\begin{proof}
 Let us consider $l$ as the length measure. Notice that if $l(P(X,N))>0$, then $l(\pi(P(X,N)))>l(P(X,N))$, but since $\pi^n(P(X,N)) \subset I$ and $l(I)<\infty$ the family $\{\pi^n(P(X,N)))\}_{n \in \mathbb N}$ cannot be pairwise disjoint, otherwise one would have $$ \infty > l(I) \geq l(\cup_n \pi^n(P(X,N)))=\sum_n \pi^n(P(X,N)) > \sum_n l(P(X,N)) = \infty,$$ which is an absurd. Therefore, there must exist  $n_1$ and $n_2$ such that $$\pi^{n_1}(P(X,N)) \cap \pi^{n_2}(P(X,N)) \neq \emptyset,$$ the above lemmas imply that $\pi^{n_2-n_1}P(X,N)=P(X,N)$, which cannot happen since $|\pi'|>1$. Hence $P(X,N)$ is a point.
\end{proof}

By Proposition \ref{prop:pi.expands} we know that if $q_1$ and $q_2$ are sufficiently close to $q_0$, then $|\pi'|>1$ on $I$. This means that the functions $h_k$ is well defined by the above lemma as
\begin{eqnarray*}\label{hk}
 h_k:\Sigma_2 \times \Sigma^*_k & \rightarrow & \Lambda \\
 (X,N) & \mapsto & P(X,N).
\end{eqnarray*}
Since the domain of $h_{k+1}$ contain the domain of $h_k$ and the two functions by construction coincide on the domain of $h_k$, the function $h$ 
\begin{eqnarray*}\label{h}
 h:\Sigma_2 \times \Sigma^b & \rightarrow & \Lambda \\
 (X,N) & \mapsto & h_k(X,N), \text{ if } (X,N) \in \Sigma_2 \times \Sigma^*_k.
\end{eqnarray*}
is well defined. Recall that $\pi(P(X,N)) = P(\sigma(X,N))$, which implies $\pi \circ h_k = h_k \circ \sigma_k$ as well as $\pi \circ h = h \circ \sigma$.

\begin{lemma} The maps $h_k$ and $h$ are continuous.
\end{lemma}
\begin{proof}
Let $(X,N)\in \Sigma_2\times \Sigma_k^*$ and $\epsilon >0$ be given. From \eqref{P} and \eqref{hk} we know that $$h_k(X,N)= \bigcap_{n \in \mathbb N} P_n(X,N),$$ hence consider $n_\epsilon$ such that $P_{n_\epsilon}(X,N) \subset (-\epsilon +h_k(X,N), h_k(X,N) + \epsilon) \subset I$. 

Let $\mathcal V_{(X,N)}$ be a neighborhood of $(X,N)$ in $\Sigma_2 \times \Sigma_k^*$ given by the cylinder
$$ \mathcal V_{(X,N)}:= \{(Y,M) | y_i = x_i, \; m_{i+1} =n_{i+1} \; i,j \in \{0,1, \ldots, n_{\epsilon}\}\}.$$
And the continuity follows, since
$$h_k(\mathcal V_{(X,N)}) \subset (-\epsilon +h_k(X,N), h_k(X,N) + \epsilon).$$

The same proof serves for $h$.
\end{proof}

\begin{lemma}
 The map $h_k$ and $h$ are homeomorphisms onto their image.
\end{lemma}
\begin{proof}
We prove $h_k$ is a homeomorphism onto its image, the case for $h$ is analogous. Notice that $h_k$ is injective by Lemma \ref{l1}. To see that the inverse is continuous, consider a point $h_k(X,N)$ and a neighborhood $\mathcal U_{(X,N)}$ of $(X,N)$, therefore $(Y,M)\in\mathcal U_{(X,N)}$ means that the the first digits of both sequences $(X,N)$ and $(Y,M)$ coincide, using the continuity of the flow we get that for points close enough to $h_k(X,N)$ they must have this predefined trajectory and the continuity follows.
\end{proof}

The above lemmas imply Theorem \ref{theo:topological}, where $\Lambda_k := h_k(\Sigma_2 \times \Sigma_k^*)$.
\hfill $\square$

\section{Final comments}

\subsection{Consequences}We are able to fully understand the dynamics around the point $q_0 \in \partial M^s$. Not only from its topological point of view as well as from the ergodic point of view. We give a few consequences of our Theorem \ref{theo:topological} below.

The first consequence we state is that if $K$ is an invariant compact set for $\pi$ (the first return map) then $\pi|\Lambda$ "lives" in some symbolic dynamics:

\begin{mcor}\label{cor:compact.invariant}
 Given a compact set $K \subset \mathcal U$ $\pi$-invariant, then $\pi|K$ is conjugate for some $k$ to a $\sigma_k|\Omega$ where $ \Omega \subset \Sigma_2 \times \Sigma_k^*$ is an $\sigma_k$-invariant set.
\end{mcor} 
\begin{proof}
Notice that for $\xi\in\Lambda,$ there exists a positive number $M_\xi<\infty$ such that $n_*(\xi)<M_\xi$ for $* \in \{0,1\}.$ By continuity of $\pi$ there is an neighbourhood $U_{\xi}$ of ${\xi}$ such that $\forall z \in U_{\xi}$ $n_*(z)\leq M_{\xi}$. Since $K$ is a compact set take a finite cover of $K$ and consider the maximum of $n_*$ for these finite cover. Let $k_0$ be this maximum. This means that for all points in $K$ if we catalogue its trajectory it has to be given by a sequence in $\Sigma_2 \times \Sigma_{k_0}^*$. Proving the corollary.
\end{proof}

We are able to fully characterize the ergodic properties of the system:

\begin{mcor}[Ergodic]\label{theo:ergodic}
 There is a neighbourhood $\mathcal U \subset \partial M$ of $q_0$ such that: if $(\pi,\mu)$ is ergodic, then there exist $k\in \mathbb N$  such that
 \begin{itemize}
  \item $\mu(\Lambda_k)=1$;
  \item there exist  a measure $\nu$ which is $\sigma_k$-invariant for which $(\pi|\Lambda_k), \mu)$ is isomorphic to $(\sigma_k, \nu)$.
 \end{itemize}
\end{mcor}
\begin{proof}
 We know that $ \Lambda_k \subset \Lambda_{k+1}$ and $\Lambda = \bigcup_{k=1}^\infty \Lambda_k$ and $\Lambda_k$ is $\pi$-invariant. Hence, by ergodicity $\mu(\Lambda_k)\in \{ 0,1\}$, if $\mu(\Lambda_k)=0$, $\forall k \in \mathbb N$ then $\mu(\Lambda)=0$ which is an absurd. Therefore, there exist $k_0$ such that $\mu(\Lambda_{k_0})=1$. Since there is a conjugacy from $\pi|\Lambda_{k_0}$ to $\sigma_{k_0}$ the theorem is done with $\nu:= (h_{k_0})_*\mu$.
\end{proof}

\subsection{Conclusion and further directions} In this paper we studied Filippov systems admitting a sliding Shilnikov orbit $\G$, which is a homoclinic connection inherent to Filippov systems. This connection has been firstly studied in \cite{NT}. Using the well known theory of Bernoulli shifts, we were able to provide a full topological and ergodic description of the dynamics of Filippov systems nearby a sliding Shilnikov orbit $\Gamma$, answering then some inquiries made in \cite{NT}. As our main result, we established the existence of a set $\Lambda\subset \p M^s$ such that the restriction to $\Lambda$ of the the first return map $\pi$, defined nearby $\G$, is topologically conjugate to a Bernoulli shift with infinite topological entropy. This ensures $\pi$, consequently the flow, to be as much chaotic as one wishes. In particular, given any natural number $m\geq 1$ one can find infinitely many periodic points of the first return map with period $m$ and, consequently, infinitely many closed orbits nearby $\G$ of the Filippov system.

As it has already been observed in \cite{NT}, a possible direction for further investigations is to consider higher dimensional vector fields, since in higher dimension it is allowed the existence of many other kinds of sliding homoclinic connections. We feel that the techniques applied in this paper may be straightly followed to obtain similar results in higher dimensions. Another open question is to know what happens when the sliding Shilnikov orbit is unfolded.

\section*{Acknowledgements}

%We thank to the referees their comments and suggestions which help
%us to improve greatly the presentation of this paper.

The authors are very grateful to Marco A. Teixeira for reading the paper and making helpful comments and suggestions.

The first author is supported by a FAPESP grant 2015/02517-6 and the second author was supported by FAPESP grant 2015/02731-8.

\end{document}